\documentclass[11pt, letterpaper]{amsart}
\usepackage{amssymb,amsthm,amsmath}

\DeclareFontFamily{OMS}{rsfs}{\skewchar\font'60}
\DeclareFontShape{OMS}{rsfs}{m}{n}{<-5>rsfs5 <5-7>rsfs7 <7->rsfs10 }{}
\DeclareSymbolFont{rsfs}{OMS}{rsfs}{m}{n}
\DeclareSymbolFontAlphabet{\scr}{rsfs}

\newtheorem{theorem}{Theorem}[section]
\newtheorem*{theorem*}{Theorem}
\newtheorem*{mainresult}{Main Result}
\newtheorem{lemma}[theorem]{Lemma}
\newtheorem{proposition}[theorem]{Proposition}
\newtheorem{corollary}[theorem]{Corollary}

\newtheorem*{conjecture*}{Conjecture}

\theoremstyle{definition}
\newtheorem{definition}[theorem]{Definition}

\theoremstyle{remark}
\newtheorem{remark}[theorem]{Remark}

\newtheorem{question}[theorem]{Question}
\newtheorem*{question*}{Question}

\newcommand{\blank}{\underline{\hskip 10pt}}

\newcommand{\m}{\scr M}
\newcommand{\n}{\scr N}

\newcommand{\N}{\mathbb {N}}

\newcommand{\Z}{\mathbb {Z}}
\newcommand{\Q}{\mathbb {Q}}

\newcommand{\tensor}{\otimes}

\DeclareMathOperator{\Frac}{{Frac}}
\DeclareMathOperator{\Ann}{{Ann}}

\DeclareMathOperator{\coker}{{coker}}

\DeclareMathOperator{\rank}{{rank}}

\DeclareMathOperator{\Spec}{{Spec}}

\DeclareMathOperator{\Hom}{Hom}

\def\spec#1.#2.{{\bold S\bold p\bold e\bold c}_{#1}#2}
\def\proj#1.#2.{{\bold P\bold r\bold o\bold j}_{#1}\sum #2}
\def\ring#1.{\scr O_{#1}}
\def\map#1.#2.{#1 \to #2}
\def\longmap#1.#2.{#1 \longrightarrow #2}
\def\factor#1.#2.{\left. \raise 2pt\hbox{$#1$} \right/
\hskip -2pt\raise -2pt\hbox{$#2$}}
\def\pe#1.{\mathbb P(#1)}
\def\pr#1.{\mathbb P^{#1}}

\def\coh#1.#2.#3.{H^{#1}(#2,#3)}
\def\dimcoh#1.#2.#3.{h^{#1}(#2,#3)}
\def\hypcoh#1.#2.#3.{\mathbb H_{\vphantom{l}}^{#1}(#2,#3)}
\def\loccoh#1.#2.#3.#4.{H^{#1}_{#2}(#3,#4)}
\def\dimloccoh#1.#2.#3.#4.{h^{#1}_{#2}(#3,#4)}
\def\lochypcoh#1.#2.#3.#4.{\mathbb H^{#1}_{#2}(#3,#4)}
\def\ses#1.#2.#3.{0  \longrightarrow  #1   \longrightarrow
 #2 \longrightarrow #3 \longrightarrow 0}
\def\sesshort#1.#2.#3.{0
 \rightarrow #1 \rightarrow #2 \rightarrow #3 \rightarrow 0}
 
\def\iff#1#2#3{
    \hfil\hbox{\hsize =#1 \vtop{\noin #2} \hskip.5cm
    \lower.5\baselineskip\hbox{$\Leftrightarrow$}\hskip.5cm
    \vtop{\noin #3}}\hfil\medskip}
\def\myoplus#1.#2.{\underset #1 \to {\overset #2 \to \oplus}}

\renewcommand{\:}{\colon}
\newcommand{\cf}{{\itshape cf.} }

\renewcommand{\m}{\mathfrak{m}}

\DeclareMathOperator{\length}{\ell}

\usepackage[left=1.25in,top=1.25in,right=1.25in,bottom=1.25in]{geometry}

\usepackage{setspace}
\usepackage{hyperref}

\usepackage{enumerate}

\usepackage{graphicx}

\usepackage[all,cmtip]{xy}
\usepackage{verbatim}

\usepackage[usenames,dvipsnames]{color}

\begin{document}

\title[$F$-signature exists]{\LARGE $F$-Signature exists}

\author{Kevin Tucker}

\thanks{This material is based upon work supported by the National
  Science Foundation under Award No. 1004344.}
\subjclass[2000]{14B05, 13A35}

\maketitle
\tableofcontents

\spacing{1.2}

\begin{abstract}
  Suppose $R$ is a $d$-dimensional reduced \mbox{$F$-finite} Noetherian local ring with
  prime characteristic \mbox{$p > 0$} and perfect residue field.  Let $R^{1/p^e}$ be the
  ring of $p^e$-th roots of elements of $R$ for $e \in \N$, and let
  $a_{e}$ denote the maximal rank of a free $R$-module appearing in a
  direct sum decomposition
  of $R^{1/p^{e}}$.  We show the existence of the limit $s(R) := \lim_{e \to \infty} \frac{a_{e}}{p^{ed}}$, called the 
  \mbox{$F$-signature} of $R$.  This invariant -- which can be
  extended to all local rings in prime characteristic -- was first formally
  defined by C. Huneke and G. Leuschke
  \cite{HunekeLeuschkeTwoTheoremsAboutMaximal} 
and has previously been shown to exist only in special cases.
The proof of our main result is based on the development
  of certain uniform
  Hilbert-Kunz estimates of independent interest.  Additionally, we
  analyze the behavior of the $F$-signature under finite ring
  extensions and recover explicit formulae for the \mbox{$F$-signatures} of
arbitrary finite quotient singularities
\end{abstract}


\section{Introduction}
\label{sec:introduction}

Every ring $R$ with prime characteristic $p > 0$ comes endowed with a
Frobenius or \mbox{$p$-th} power endomorphism.   
The existence of an $R$-module section of Frobenius, called an
 \mbox{\emph{$F$-splitting}}, has strong algebraic and geometric
consequences.
Historically, $F$-splittings have featured prominently
throughout diverse fields of mathematics, 
and applications of these techniques include the well-known theorems of M. Hochster and
J. L. Roberts \cite{HochsterRobertsRingsOfInvariants} together with
numerous results in representation theory
\cite{BrionKumarFrobeniusSplitting}.  In this paper, we
answer an important question which has remained open for over a decade by showing the
existence of a local numerical invariant -- the $F$-signature -- which
roughly
characterizes the asymptotic growth of the number of splittings of
the iterates of Frobenius.

More precisely, let $R$ be complete $d$-dimensional reduced Noetherian local ring with prime characteristic
$p > 0$ and perfect residue field $k = k^{p}$.  For  $e \in \N$ the
inclusion $R \subseteq R^{1/p^{e}}$ into the corresponding ring of
$p^{e}$-th roots of elements of $R$ is naturally identified with the
$e$-th iterate of the Frobenius endomorphism.
Let $a_{e}$
denote the largest rank of a free \mbox{$R$-module} appearing in a direct sum
decomposition 
of $R^{1/p^{e}}$.  In other words, we may write $R^{1/p^{e}} =
R^{ \oplus a_{e}} \oplus M_{e}$ as $R$-modules where $M_{e}$ has no free
direct summands.
The number $a_{e}$ is called the \emph{$e$-th Frobenius splitting
  number} of $R$, and collectively these numbers
encode
subtle information about the action of the Frobenius endomorphism on
$R$.   

The primary goal of this paper is to analyze the asymptotic
behavior of the sequence $\{ a_{e} \}_{e \in \N}$ by showing the existence of the limit
$s(R) = \lim_{e \to \infty} \frac{a_{e}}{p^{ed}}$,
called the \emph{$F$-signature~of~$R$}.
\begin{mainresult}[Theorem~\ref{thm:mainthm}] The $F$-signature $s(R)
  := \lim_{e \to \infty} \frac{a_{e}}{p^{ed}}$ exists.
\end{mainresult}
\noindent
This invariant was first formally defined by C. Huneke and G. Leuschke
\cite{HunekeLeuschkeTwoTheoremsAboutMaximal} and captures delicate information about the
singularities of $R$. For example, the $F$-signature of 
the two-dimensional rational double-points\footnote{Here it is
  necessary to assume that $p \geq 7$ to avoid pathologies in low characteristic.} ($A_{n}$), ($D_{n}$),
($E_{6}$), ($E_{7}$), ($E_{8}$) is the reciprocal of the order of the group
defining the corresponding quotient singularity \cite[Example 18]{HunekeLeuschkeTwoTheoremsAboutMaximal}.
We recover herein similar formulae for the \mbox{$F$-signatures} of
arbitrary finite quotient singularities (Corollary~\ref{cor:quotsing};
\textit{c.f.} \cite[Remark 4.7]{YaoObservationsAboutTheFSignature}).

It is quite natural to expect the 
\mbox{$F$-signature} to measure the singularities of $R$.  Indeed,
when $R$ is regular, $R^{1/p^{e}}$ itself is a free $R$-module of rank
$p^{ed}$.  Thus, for general $R$, the \mbox{$F$-signature} asymptotically
compares the number of direct summands of $R^{1/p^{e}}$ isomorphic to $R$ with
the number of such summands one would expect from a regular local ring
of the same dimension.  We will see (\cf Theorem~\ref{thm:bigsmeansregular}) that $s(R) \leq 1$ with equality
if and only if $R$ is regular; furthermore, assuming its existence, 
it has been shown by I. Aberbach
and G. Leuschke \cite{AberbachLeuschke} that the \mbox{$F$-signature} is positive if and only if $R$
is strongly $F$-regular.

In light of lingering doubts concerning existence,  previous
research on the \mbox{$F$-signature} has largely been done through the
use of
so-called lower and upper \mbox{$F$-signatures}.  These
are given by $s^{-}(R) :=\liminf_{e \to \infty}
\frac{a_{e}}{p^{ed}}$ and $s^{+}(R) := \limsup_{e \to
  \infty} \frac{a_{e}}{p^{ed}}$, respectively.  Thus, the goal of this paper is simply to
show the equality $s^{-}(R) = s^{+}(R)$ of the lower and upper
\mbox{$F$-signatures} in full generality.

Historically, the $F$-signature first appeared implicitly
in the work of K. Smith and M. Van den Bergh
\cite{SmithVanDenBerghSimplicityOfDiff}.  However, since the beginning
of its formal study in \cite{HunekeLeuschkeTwoTheoremsAboutMaximal},
the existence of the  \mbox{$F$-signature} limit has 
been shown only in special cases.  These include Gorenstein local
rings \cite{HunekeLeuschkeTwoTheoremsAboutMaximal}, local rings that are
Gorenstein on the punctured spectrum \cite{YaoObservationsAboutTheFSignature}, affine
semigroup rings \cite{SinghFSignatureOfAffineSemigroup}, general
$\N$-graded rings \cite{AberbachEnescuWhenDoesTheFSignatureExist}, and
local rings that are
$\Q$-Gorenstein on the punctured spectrum \cite{AberbachEnescuWhenDoesTheFSignatureExist}.
Most recently, I. Aberbach \cite{AberbachExistenceOfTheFSignature} uses certain degree bounds on local cohomology
modules\footnote{These bounds are known to hold for rings which are essentially of
finite type over a field \cite{VraciuLocalCohomofFrobImages}.} to treat the case of rings whose
\mbox{non-$\Q$-Gorenstein} locus has dimension at most one.
It should be noted that the proof given
herein will use rather elementary techniques in comparison.

\nocite{AberbachExistenceOfTheFSignature,AberbachEnescuWhenDoesTheFSignatureExist,AberbachEnescuStructureOfFPure,AberbachLeuschke,YaoObservationsAboutTheFSignature,SinghFSignatureOfAffineSemigroup,EnescuYao,YaoModuleswithfiniteFrepresentationtype}

Let us sketch the proof of the existence of the \mbox{$F$-signature}.  Recall
that, according to
the famous result of P. Monsky \cite{MonskyHKFunction}, for any $\m$-primary ideal $I =
\langle x_{1}, \ldots, x_{r} \rangle$ we may
define the
\emph{Hilbert-Kunz multiplicity of $I$ along R}
\[
e_{HK}(I;R) := \lim_{e\to\infty} \frac{1}{p^{ed}} \length(R/I^{[p^{e}]})
\]
where $\length(\blank)$ denotes the length function on Artinian $R$-modules and
$I^{[p^{e}]} := \langle x_{1}^{p^{e}}, \ldots, x_{r}^{p^{e}}
\rangle$. In Section~\ref{sec:HK-mult}, we give a variant on the
original proof of the existence of Hilbert-Kunz multiplicity which
carefully tracks certain uniform estimates.  The most important of
these is the following (Proposition~\ref{prop:keyproposition}):
for any $d$-dimensional reduced \mbox{$F$-finite} ring $(R, \m, k)$,
there exists a positive constant $C$ such that for all $e, e' \in \N$
and every ideal $I$ of $R$ containing $\m^{[p^{e}]}$ we have
\[
\left| \length(R/I) - \frac{1}{p^{e'd}}\length(R/I^{[p^{e'}]})\right|
\leq Cp^{e(d-1)} \, \, .
\]

In Section~\ref{sec:f-signature}, building on the works of Y. Yao \cite{YaoObservationsAboutTheFSignature} as well as F. Enescu and
I. Aberbach \cite{AberbachEnescuStructureOfFPure}, for each fixed $e
\in \N$ we consider the ideal
\[
I_{e} = \{ r \in R \, | \, \phi(r^{1/p^{e}}) \in \m \mbox{ for all } \phi \in
\Hom_{R}(R^{1/p^{e}}, R) \} \, \, .
\]
It is easy to see that $\m^{[p^{e}]} \subseteq I_{e}$, so it follows
(Corollary~\ref{cor:whatineed}) by the uniform estimate above that
\[
\lim_{e \to \infty} \left( \frac{1}{p^{ed}} \left(\length(R/I_{e}) -
    e_{HK}(I_{e};R)\right)\right) = 0 \, \, .
\]
Since one can show 
$a_{e} = \length(R/I_{e})$
for all $e \in \N$, to prove the existence of the \mbox{$F$-signature} it suffices to show the sequence $\{ \frac{1}{p^{ed}} e_{HK}(I_{e};R)
\}_{e\in \N}$ approaches a limit. This follows by noting that $I_{e}^{[p]} \subseteq
I_{e+1}$ for all $e \in \N$, and thus 
$\{ \frac{1}{p^{ed}} e_{HK}(I_{e};R)
\}_{e\in \N}$
 is non-increasing (and bounded below by zero).  Note that, in comparison to 
 previous existence arguments, we do not realize the
 \mbox{$F$-signature} as a relative Hilbert-Kunz multiplicity
 (\cf Remark~\ref{rmk:minrelHKmult});
 rather, it is a limit of (appropriately scaled) Hilbert-Kunz
 multiplicities of a sequence of naturally defined ideals.

In his work \cite{YaoObservationsAboutTheFSignature}, Y. Yao has
generalized the \mbox{$F$-signature} to arbitrary local rings $(R, \m, k)$ in prime
characteristic $p > 0$ without the simplifying assumptions that $R$ is
complete and the residue field $k$ is perfect.
The existence of the \mbox{$F$-signature} in full generality, however,
immediately reduces to the case of \mbox{$F$-finite} local rings
originally considered in \cite{HunekeLeuschkeTwoTheoremsAboutMaximal} (\cf Remark~\ref{rmk:notFfinite}). 
As such, we have
largely restricted ourselves to the \mbox{$F$-finite} setting
throughout.  Under this hypothesis,  one incorporates $\alpha(R) = \log_{p} [k:k^{p}]$ into the definition
$\lim_{e \to \infty}
\frac{a_{e}}{p^{e(d+\alpha(R))}}$ of the $F$-signature when the residue field $k$ is not
perfect \cite{AberbachLeuschke}.

It has been noted that in almost all previously known
cases of the existence of \mbox{$F$-signature}, the conjectured
equivalence of strong and weak
$F$-regularity holds.  This observation has lead to
much speculation concerning a connection between this conjecture and the existence of 
\mbox{$F$-signature};
however, we are as of yet unaware of an application of our results or
methods in this direction (\cf Remark \ref{rmk:minrelHKmult}).

Much of the study of \mbox{$F$-signature} to date has focused simply on the
existence of this invariant.  With this chapter closed, however, we
would argue that the subject is now ripe for new exploration.  It is our hope that  this is but another beginning in the use of
\mbox{$F$-signature} to better understand local rings in
positive characteristic. 

The author would like to thank Manuel Blickle and Karl Schwede for
discussions and encouragement related to this article, as well as
Craig Huneke for sharing an insightful observation (\cf Remark~\ref{rem:duttacraig}) after seeing a
preliminary draft.

\section{Background and notation}
\label{sec:background-notation}

Throughout this paper, we shall assume all rings are commutative with
a unit, 
Noetherian, and have
prime characteristic $p > 0$.  A
local ring is a triple $(R, \m, k)$ where $\m$ is the unique maximal
ideal of the ring $R$
and $k = R/\m$ is the corresponding residue field.  The Frobenius
or $p$-th power endomorphism $F \: R \to R$ is defined by $r \mapsto
r^{p}$ for all $r \in R$.  Similarly, for $e \in \N$, we have $F^{e}
\: R \to R$ given by $r \mapsto r^{p^{e}}$.

Let $M$ be an $R$-module.  For any $e \in \N$, viewing $M$ as an
$R$-module via restriction of scalars under $F^{e}$ yields an
$R$-module we  denote by
$F^{e}_{*}M$.  Thus, $F^{e}_{*} M$ agrees with $M$ as an
abelian group, and if $m \in M$ we set $F^{e}_{*}m$ to be the
corresponding element of $F^{e}_{*}M$.  Furthermore, for $r \in R$ it
follows that $r 
(F^{e}_{*}m) = F^{e}_{*}(r^{p^{e}}  m)$.  Note that $F^{e}_{*}R$
inherits the structure of a ring abstractly isomorphic to $R$, and $F^{e}_{*}M$ is naturally an
$F^{e}_{*}R$-module for any $R$-module $M$. 

We have that $F^{e}_{*}R$
is an $R$-algebra via the homomorphism of $R$-modules $F^{e} \: R \to
F^{e}_{*}R$ given by $r
\mapsto F^{e}_{*} r^{p^{e}}$ for $r \in R$, which is but another
perspective on the $e$-th iterate of Frobenius.
In case $R$ is reduced, we may
identify $F^{e}_{*}R$ with the $R$-module $R^{1/p^{e}}$ of $p^{e}$-th
roots of $R$ by associating  $F^{e}_{*}r$ and $r^{1/p^{e}}$; the $e$-iterated Frobenius homomorphism
now takes on the
guise of the natural inclusion $R \subseteq R^{1/p^{e}}$.  Each point
of view has certain advantages, and we will
switch between them as the situation warrants throughout.

\begin{definition}
  Suppose $(R, \m, k)$ is a local ring of characteristic $p > 0$.  We
  say $R$ is \mbox{\emph{$F$-finite}} if $F_{*}R$ is finitely generated as an
  $R$-module, from which it follows that $F^{e}_{*}R$ is finitely
  generated for all $e \in \N$.  In this case, we set $\alpha(R) = \log_{p}[k:k^{p}]$.
\end{definition}

\noindent
Note that any local ring which is essentially of finite type over a
perfect field is \mbox{$F$-finite}. Additionally, as used for
simplicity in the introduction, a complete local
ring with \mbox{$F$-finite} residue field is
automatically \mbox{$F$-finite}.

Denote by $\length_{R}(M)$ the length of a finitely generated
Artinian $R$-module $M$.  If $R$ is \mbox{$F$-finite} and $e \in \N$, it is easy to
see that
\[
\length_{R}(M) = \length_{F^{e}_{*}R}(F^{e}_{*}M)  \qquad
\length_{R}(F^{e}_{*}M) = p^{e\alpha(R)}\length_{R}(M)
\]
by using that $F^{e}_{*}( \blank )$ is an exact functor and
$[(F^{e}_{*}k \simeq k^{1/p^{e}}):k] = p^{e\alpha(R)}$.  The following Theorems of Kunz shall
be used repeatedly.  A secondary reference for these
Theorems can be found in the appendix to \cite{MatsumuraCommutativeAlgebra}.

\begin{theorem}(Kunz's Theorems)
\label{thm:kunz}  Let $(R, \m, k)$ be a Noetherian local ring of dimension $d$ and characteristic $p >
  0$.
  \begin{enumerate}[(i)]
  \item \cite{KunzCharacterizationsOfRegularLocalRings} For any $e \in
    \N$, we have $\length(R/\m^{[p]}) \geq p^{ed}$
    and equality holds if and only if $R$ is regular.  If $R$ is
    \mbox{$F$-finite}, then $R^{1/p^{e}}$ is a free module of rank
    $p^{e(d+\alpha(R))}$ if and only if $R$ is regular.
  \item \cite{KunzOnNoetherianRingsOfCharP}\footnote{We caution the
      reader that Kunz states a nearby and related result in more generality
      than his proof justifies.  See \cite[page 4]{EnescuYao} for
      further details (\cf \cite{ShepherdBarronOnaProblemofKunz}).} If $R$ is \mbox{$F$-finite}, then $R$ is excellent and $\alpha(R_{P})
    = \alpha(R_{Q}) + \dim(R_{Q} / PR_{Q})$ for any two prime ideals
    $P \subseteq Q$ of $R$.
  \end{enumerate}
\end{theorem}

We close this section by recalling the definitions of $F$-purity and
strong $F$-regularity for \mbox{$F$-finite} local rings.  Though first
introduced in \cite{HochsterRobertsFrobeniusLocalCohomology} and
\cite{HochsterHunekeTightClosureAndStrongFRegularity}, respectively,
both concepts have been studied by many authors.  As little of
the theory will be used in subsequent sections, we content
ourselves to recall a few standard facts without proofs or references.
The interested reader is invited to see \cite{HunekeTightClosureBook}
or \cite{HochsterFoundations} for further information.

\begin{definition}
  Suppose $(R, \m, k)$ is an \mbox{$F$-finite} local ring with prime characteristic
  $p>0$. 
\begin{itemize}
\item We say $R$ is \emph{$F$-pure} or\footnote{While technically speaking
    $F$-purity is an a priori weaker condition than $F$-splitting, the
  two notions coincide for $F$-finite rings; see \cite{HochsterRobertsFrobeniusLocalCohomology}.} \emph{$F$-split} if
the Frobenius homomorphism $F \: R
  \to F_{*}R$ splits as a map of $R$-modules.  In other words, there exists $\phi \in \Hom_{R}(F_{*}R,R)$ such that $\phi \circ
    F = \mathrm{Id}_{R} \in \Hom_{R}(R,R)$.  In case $R$ is $F$-pure, it is
    automatically reduced and weakly normal, and the $\m$-adic
    completion $\widehat{R}$ is also $F$-pure.
  \item Let $R^{0}$ be the complement of the minimal primes of $R$.
    We say $R$ is
  \emph{strongly \mbox{$F$-regular}} if
for every $c \in R^{0}$ there exists an $e \geq 0$ and some
    $\phi \in \Hom_{R}(F^{e}_{*}R, R)$ such that $\phi(F^{e}_{*}c) =
    1$.  In other words
the $R$-linear map $R \to F^{e}_{*}R$ which sends $1$ to $F^{e}_{*}c$
splits over $R$.  In case $R$ is strongly $F$-regular, it is a
Cohen-Macaulay normal domain, and the $\m$-adic completion
$\widehat{R}$ is also strongly $F$-regular.
  \end{itemize}
\end{definition}

\begin{remark}
The notions of $F$-purity and $F$-regularity play a prominent role in the celebrated
theory of tight closure introduced by M. Hochster and C. Huneke; see 
  \cite{HochsterHunekeTC1} for a first glance at this beautiful subject.
They have conjectured that all ideals of $R$ are tightly closed if and
only if $R$ is strongly $F$-regular.  This is known to be true 
when $R$ is an excellent $\Q$-Gorenstein normal local ring
(\cite{AberbachMacCrimmonSomeResultsOnTestElements}; \cf
\cite{LyubeznikSmithStrongWeakFregularityEquivalentforGraded,LyubeznikSmithCommutationOfTestIdealWithLocalization} and
\cite[Theorem 1.13]{HaraYoshidaGeneralizationOfTightClosure}).
\end{remark}

\section{Uniform Hilbert-Kunz estimates}
\label{sec:HK-mult}

Our goal in this section is to revisit the proof
of a famous
result of P. Monsky.  For an ideal $I$ of a local ring $(R,\m,k)$ and $e
\in \N$, recall
that $I^{[p^{e}]}$ denotes the ideal $\langle x^{p^{e}} \, | \, x \in
I \rangle$.  Furthermore, we have 
\[R/I \tensor_{R} F^{e}_{*}R =
F^{e}_{*}R/I\, F^{e}_{*}R = F^{e}_{*}R/F^{e}_{*}I^{[p^{e}]} =
F^{e}_{*}(R/I^{[p^{e}]})\] by the definition of the action of $R$ on $F^{e}_{*}R$. 

\begin{theorem} \cite{MonskyHKFunction}
\label{thm:hkexists}
Suppose $(R,\m,k)$ is a local ring of dimension
$d$ and characteristic $p>0$.  If $I$ is any $\m$-primary ideal and
$M$ is a finitely generated $R$-module, then
the limit 
\[
e_{HK}(I;M) := \lim_{e\to\infty}\frac{1}{p^{ed}} \length_{R}(R/I^{[p^{e}]} \tensor_{R}M)
\]
exists and is called the \emph{Hilbert-Kunz multiplicity of $M$ along $I$}.
\end{theorem}

The proof of Theorem~\ref{thm:hkexists} given here is but a slight variant
of Monsky's original proof.  However, in the process, we will recover
certain uniform approximation statements (see Theorem~\ref{thm:approximation}) which will be essential in
showing the existence of the \mbox{$F$-signature}.  
We begin with a pair of rather elementary and well-known
lemmas. 

\begin{lemma}
\label{lem:lowerdim}
  Suppose $(R,\m,k)$ is a local ring of characteristic $p>0$.  If 
  $M$ is a finitely generated $R$-module, then
  there exists a positive constant $C$ such that for all $e \in \N$
  and any ideal $I$ of $R$ with $\m^{[p^{e}]}\subseteq I$ we have
\[
\length_{R}(R/I \tensor_{R} M) \leq Cp^{e\dim(M)} \, \, .
\]
\end{lemma}

\begin{proof}
Since $R/\m^{[p^{e}]} \tensor_{R} M$ surjects onto $R/I
\tensor_{R} M$, it suffices to show the statement for $I = \m^{[p^{e}]}$.
If $\m$ has $t$ generators, then $\m^{tp^{e}} \subseteq \m^{[p^{e}]}$ and
hence
\[
\length_{R}(R/\m^{[p^{e}]} \tensor_{R} M) \leq \length_{R}(R/(\m^{t})^{p^{e}} \tensor_{R}M)
\, \,.
\]  The Hilbert polynomial of $M$ with respect to $\m^{t}$ has degree
$\dim(M)$.  If 
leading coefficient of this polynomial is $c$, it is clear that any $C \gg c$ satisfies the desired bound.
\end{proof}

\begin{lemma}
\label{lem:second}
  Suppose $(R,\m,k)$ is a $d$-dimensional reduced local ring of characteristic $p > 0$.  Let
  $P_{1}, \ldots, P_{m}$ be those minimal primes of $R$ with
  $\dim(R/P_{i}) = d$.  If $M$ and $N$ are finitely
  generated $R$-modules such that $M_{P_{i}} \simeq N_{P_{i}}$ for
  each $i$, then there exists a positive constant $C$ such that for
  all $e \in \N$ and any ideal $I$ of $R$ with $\m^{[p^{e}]}\subseteq I$ we have
\[
\left| \length_{R}( R/I \tensor_{R} M ) - \length_{R}( R/I \tensor_{R} N)
\right| \leq C p^{e(d-1)} \, \, .
\]
\end{lemma}

\begin{proof}
Let $W = R \setminus (\cup_{i} P_{i})$, so that $W^{-1}R = R_{P_{1}}
\times \cdots \times R_{P_{m}}$ and we have $W^{-1}M \simeq W^{-1}N$.
Since $W^{-1} \Hom_{R}(M,N) = \Hom_{W^{-1}R}(W^{-1}M, W^{-1}N)$, there
is some $\phi \in \Hom_{R}(M,N)$ such that $W^{-1}\phi$ is an
isomorphism. Since $\coker(\phi)$ satisfies $W^{-1}\coker(\phi) = 0$ and
thus has dimension
strictly smaller than $d$, we can find a positive constant $C$
such that for all $e \in \N$ and any ideal $I$ of $R$ containing
$\m^{[p^{e}]}$ we have
\[
\left| \length_{R}(R/I \tensor_{R} \coker(\phi) )\right| \leq
Cp^{e(d-1)} \, \, .
\] The
sequence
\[
\xymatrix{R/I \tensor_{R} M \ar[r]^{\phi} & R/I \tensor_{R} N \ar[r] & R/I
\tensor_{R} \coker(\phi) \ar[r] & 0}
\]
is exact, so that
\[
\length_{R}(R/I \tensor_{R} N) - \length_{R}(R/I \tensor_{R} M) \leq
\length_{R}(R/I \tensor_{R} \coker(\phi)) \, \, .
\]
The lemma now follows by reversing the roles of $M$ and $N$ and
applying the preceding lemma.
\end{proof}

The next proposition is the key result from which both
the existence of Hilbert-Kunz multiplicity and $F$-signature will
follow.  The essential point is that the constant $C$ depends only on
the $R$-module $M$ in question.

\begin{proposition}
\label{prop:keyproposition}
  Suppose $(R,\m,k)$ is a $d$-dimensional reduced \mbox{$F$-finite} local ring of characteristic $p>0$.
  If $M$ is a finitely generated $R$-module, then there exists a
  positive constant $C$ such that for all $e, e' \in \N$ and any ideal $I$ of $R$ with $\m^{[p^{e}]}\subseteq I$ we have
\begin{equation}
\label{eq:4}
\left| \length_{R}\left(R/I \tensor_{R} M\right) -
  \frac{1}{p^{e'd}}\length_{R}(R/I^{[p^{e'}]} \tensor_{R} M)\right|
\leq Cp^{e(d-1)} \, \, .
\end{equation}
\end{proposition}

\begin{proof}
If $P_{1}, \ldots, P_{m}$ are the minimal primes of $R$ with
$\dim(R/P_{i}) = d$ and $K_{i} = R_{P_{i}}$, then we have by
Theorem~\ref{thm:kunz} {\itshape (ii)} that $\alpha(R_{P_{i}})=d + \alpha(R)$.
We claim that $\oplus_{p^{d+\alpha(R)}} M$ and $F_{*}M$
are isomorphic after localizing at any of the $P_{i}$.  Since $K_{i}$
is a field, this follows from
\[
\length_{R_{P_{i}}}((F_{*}M)_{P_{i}}) =
\length_{R_{P_{i}}}(F_{*}(M_{P_{i}})) = p^{\alpha(R_{P_{i}})}
\length_{R_{P_{i}}}(M_{P_{i}}) \, \, .
\]
Thus, by
Lemma~\ref{lem:second}, there is a positive constant $D$ such that for
all $e \in \N$ and any ideal $I$ of $R$ containing $\m^{[p^{e}]}$ we have
\[
\left|   \length_{R}(R/I \tensor_{R} F_{*}M) -
  p^{d+\alpha(R)}\length_{R}(R/I \tensor_{R}M)
\right| \leq D p^{e(d-1)} \, \, .
\]
Since 
\[\length_{R}(R/I \tensor_{R} F_{*}M) =
\length_{R}\left(F_{*}\left(R/I^{[p]}\tensor_{R}M\right)\right) =
p^{\alpha(R)} \length_{R}(R/I^{[p]}\tensor_{R}M)
\]
we have, setting $E = Dp^{-\alpha(R)}$,
\[
\left| \length_{R}(R/I^{[p]} \tensor_{R} M) - p^{d} \length_{R}(R/I
  \tensor_{R} M) \right| \leq Ep^{e(d-1)}
\]
for all ideals $I$ of $R$ with $\m^{[p^{e}]} \subseteq I$.  Let us now
show by induction on $e' \in \N$ that
\begin{equation}
\label{eq:1}
\left| \length_{R}(R/I^{[p^{e'}]} \tensor_{R} M) - p^{e'd} \length_{R}(R/I
  \tensor_{R} M) \right| \leq Ep^{(e+e'-1)(d-1)}(1+p + \cdots + p^{e'-1})
\, \, .
\end{equation}
Indeed, we have
\begin{equation}
\label{eq:2}
\left| \length_{R}(R/I^{[p^{e'}]} \tensor_{R} M) - p^{d}
  \length_{R}(R/I^{[p^{e'-1}]} \tensor_{R} M) \right| \leq Ep^{(e+e'-1)(d-1)}
\end{equation}
since $\m^{[p^{e+e'-1}]} \subseteq I^{[p^{e'-1}]}$.  By the induction
assumption, we have
\begin{equation*}
  \left| \length_{R}(R/I^{[p^{e'-1}]} \tensor_{R} M) - p^{(e'-1)d} \length_{R}(R/I
  \tensor_{R} M) \right| \leq Ep^{(e+e'-2)(d-1)}(1+p + \cdots +
p^{e'-2})
\end{equation*}
and multiplying through by $p^{d} = p^{d-1}p$ gives
\begin{equation}
\label{eq:3}
  \left| p^{d}\length_{R}(R/I^{[p^{e'-1}]} \tensor_{R} M) - p^{e'd} \length_{R}(R/I
  \tensor_{R} M) \right| \leq Ep^{(e+e'-1)(d-1)}(p + \cdots +
p^{e'-1}) \, \, .
\end{equation}
Adding \eqref{eq:2} and \eqref{eq:3} together completes the induction
and yields \eqref{eq:1}.

To finish the proof, dividing \eqref{eq:1} through by $p^{e'd}$ shows that
\begin{eqnarray*}
\left| \length_{R}\left(R/I \tensor_{R}M\right) -
  \frac{1}{p^{e'd}}\length_{R}(R/I^{[p^{e'}]} \tensor_{R} M)\right| & \leq &
\frac{E\left(\sum_{i=0}^{e'-1}p^{i}\right)
  p^{(e+e'-1)(d-1)}}{p^{e'd}}
\\
& = &E
\cdot \frac{p^{e'} - 1}{p - 1} \cdot \frac{p^{e(d-1)}}{p^{e'}} \cdot
\frac{p^{(e'-1)(d-1)}}{p^{e'(d-1)}}\\ 
& \leq & \left( \frac{E}{(p-1)p^{d-1}}\right) p^{e(d-1)}
\end{eqnarray*}
so that $C := \left( \frac{E}{(p-1)p^{d-1}}\right)$ is a positive
constant satisfying the desired bound.
\end{proof}

\begin{corollary}
\label{cor:keylemmaanyring}
  Suppose $(R,\m,k)$ is any local ring of dimension $d$ and
  characteristic $p > 0$.  Then there is an $e_{0} \in \Z_{\geq 0}$ (depending
  only on $R$) with the following property:
 for every finitely generated $R$-module $M$,
  there exists a positive constant $C$ so that for all $e, e' \in \N$ and any ideal $I$ of $R$ with $\m^{[p^{e}]}\subseteq I$ we have
\begin{equation}
\label{eq:5}
\left| \length(R/I^{[p^{e_{0}}]} \tensor_{R} M) -
  \frac{1}{p^{e'd}}\length(R/I^{[p^{e'+e_{0}}]} \tensor_{R}M) \right|
\leq C p^{e(d-1)} \, \, .
\end{equation}
Furthermore, if $R$ is reduced
  and \mbox{$F$-finite}, we may take $e_{0} = 0$.
\end{corollary}

{\renewcommand{\n}{\mathfrak{n}}

\begin{proof}
Let us first reduce to the case where $R$ is $F$-finite. After picking a
  coefficient field for $\widehat{R}$ and generators $x_{1}, \ldots,
  x_{n}$ for $\m$,
let $(S, \n, l)$ be the complete faithfully
flat\footnote{The ring extension $k[T_{1}, \ldots, T_{n}] \subseteq
  k^{\infty}[T_{1}, \ldots, T_{n}]$ is flat.  Since, by \cite[Chapter III \S 5]{Bourbaki1998} (\cf
  \cite[Theorem 49]{MatsumuraCommutativeAlgebra}), the
  flatness of a local homomorphism of local rings is
  preserved after completion, we see that $k[[T_{1}, \ldots, T_{n}]]
  \subseteq k^{\infty}[[T_{1}, \ldots, T_{n}]]$ is flat as well.  Thus, the
given extension $R \to S$ is faithfully flat since it is local, completion is
flat, and flatness is stable under arbitrary base change.} local ring extension of $(R, \m,
  k)$ with $\m S = \n$ and $l = l^{p}$ given by
\[
R \to \widehat{R} \to \widehat{R} \tensor_{k[[T_{1}, \ldots, T_{n}]]}
k^{\infty}[[T_{1}, \ldots, T_{n}]] =: S
\]
where $k[[T_{1}, \ldots, T_{n}]] \twoheadrightarrow \widehat{R}$ maps
$T_{i} \mapsto x_{i}$ and $l = k^{\infty}$ is a perfect closure of
$k$.
\renewcommand{\S}{S_{\mathrm{red}}}
If $S_{\mathrm{red}}= S / \mathrm{Nil}(S)$ where $\mathrm{Nil}(S)$ is
the nilradical of $S$,
  then $S_{\mathrm{red}}$ is reduced and \mbox{$F$-finite} of dimension $d$.
  Choose $e_{0} \gg 0$ so that
  $\left(\mathrm{Nil}(S)\right)^{[p^{e_{0}}]} = 0$.

If $\m^{[p^{e}]} \subseteq I \subseteq R$, then we have $(\n
S_{\mathrm{red}})^{[p^{e}]} \subseteq IS_{\mathrm{red}}$.  By Proposition~\ref{prop:keyproposition}, there is a positive constant
$C$ so that 
\[
\left| \length_{S_{\mathrm{red}}} (\S/I \S \tensor_{R} F^{e_{0}}_{*}M)
  - \frac{1}{p^{e'd}}  \length_{S_{\mathrm{red}}} (\S/I^{[p^{e'}]} \S
  \tensor_{R} F^{e_{0}}_{*}M) \right| \leq C p^{e(d-1)} \, \, .
\]
Since we have
\[
\length_{R}(R/I^{[p^{e_{0}}]} \tensor_{R} M) =
\length_{S_{\mathrm{red}}} (\S/I \S \tensor_{R} F^{e_{0}}_{*}M)
\]
\[
\length_{R}(R/I^{[p^{e'+e_{0}}]} \tensor_{R}M) = \length_{S_{\mathrm{red}}} (\S/I^{[p^{e'}]} \S
  \tensor_{R} F^{e_{0}}_{*}M)
\]
the desired result now follows.
\end{proof}
}

\begin{proof}[Proof of Theorem~\ref{thm:hkexists}]
  Fix $e_{0}$ and $C$ as in Corollary~\ref{cor:keylemmaanyring}, and
  $\epsilon > 0$. Find $e_{1} \gg 0$ so that $\m^{[p^{e_{1}}]}
  \subseteq I$.
  Choose $E \gg 0$ so that $\frac{C p^{e_{1}}}{p^{e_{0}d}p^{E}} <
  \epsilon$.  If $e \geq E$, then the estimate in \eqref{eq:5} for the ideal
  $I^{[p^{e}]} \supseteq \m^{[p^{e_{1}+e}]}$ and any $e' \in \N$ after dividing by
  $p^{(e+e_{0})d}$ yields
\[
\left| \frac{1}{p^{(e+e_{0})d}}\length_{R}(R/I^{[p^{e+e_{0}}]} \tensor_{R}
M) - \frac{1}{p^{(e'+e+e_{0})d}} \length_{R}(R/I^{[p^{e'+e+e_{0}}]}
\tensor_{R} M) \right|
\]
\[
\leq \frac{Cp^{(e+e_{1})(d-1)}}{p^{(e+e_{0})d}} =
\frac{Cp^{e_{1}(d-1)}}{p^{e_{0}d}p^{e}} \leq
\frac{Cp^{e_{1}(d-1)}}{p^{e_{0}d}p^{E}} < \epsilon
\]
In particular, this shows the sequence$\{ \frac{1}{p^{ed}}\length_{R}(R/I^{[p^{e}]}
\tensor_{R} M)\}_{e \in \N}$ is Cauchy.
\end{proof}

We view the next theorem as a kind of uniform approximation statement
for Hilbert-Kunz multiplicities.  The subsequent corollary is the
precise statement which will be needed to show the existence of the $F$-signature.

\begin{theorem}
\label{thm:approximation}
    Suppose $(R,\m,k)$ is any local ring of dimension $d$
  and characteristic $p>0$.
Then there is an $e_{0} \in \Z_{\geq 0}$ (depending
  only on $R$) with the following property:
 for every finitely generated $R$-module $M$,
  there exists a positive constant $C$ so that for all $e \in \N$ and any ideal $I$ of $R$ with $\m^{[p^{e}]}\subseteq I$ we have
\[
\left|\frac{1}{p^{e_{0}d}}\length_{R}(R/I^{[p^{e_{0}}]} \tensor_{R} M) - e_{HK}(I ; M) \right| \leq
Cp^{e(d-1)}
\]
Furthermore, if $R$ is reduced
  and \mbox{$F$-finite}, we may take $e_{0} = 0$.
\end{theorem}

\begin{proof}
Letting $e' \to \infty$ in Corollary~\ref{cor:keylemmaanyring} gives
\[\left| \length(R/I^{[p^{e_{0}}]} \tensor_{R} M) -
  e_{HK}(I^{[p^{e_{0}}]};M) \right|
\leq C p^{e(d-1)} \, \, .\]
 After dividing
through by $p^{e_{0}d}$, the desired result now follows immediately from 
\[e_{HK}(I^{[p^{e_{0}}]};M) = p^{e_{0}d}e_{HK}(I;M)\] after replacing $C$ with $\frac{1}{p^{e_{0}d}}C$.
\end{proof}

\begin{corollary}
\label{cor:whatineed}
  Suppose $(R, \m, k)$ is a $d$-dimensional \mbox{$F$-finite} reduced
  local ring of characteristic $p >
  0$, $M$ is a finitely generated $R$-module, and $\{I_{e}\}_{e\in \N}$ is any sequence of ideals such that
  $\m^{[p^{e}]} \subseteq I_{e}$.  Then
\[
\lim_{e \to \infty} \frac{1}{p^{ed}}\left( \length_{R}(R/I_{e}
  \tensor_{R} M) -
e_{HK}(I_{e};M)\right)  = 0 \, \, .
\]
\end{corollary}

\begin{remark}
\label{rem:duttacraig}
The uniform Hilbert-Kunz estimates of this
section can also be shown using ideas
from~\cite{DuttaFrobMult}. More precisely,  suppose $R$ is an $F$-finite local domain of characteristic~$p>0$ with
  dimension $d$ having perfect residue field $k$.  Then there exists a
  positive constant $C$ and a finite set $\Lambda$  of nonzero prime ideals of
  $R$ with the following property:
  \begin{quote}
For all $e \in
  \N$ there are free $R$-modules $F_{e}$ and $G_{e}$ of rank $p^{ed}$
  together with inclusions
\[
F_{e} \subseteq R^{1/p^{e}} \subseteq G_{e}
\]
such that each of the quotients $G_{e} / R^{1/p^{e}}$ and
$R^{1/p^{e}}/F_{e}$ has a prime cyclic filtration by at most $C p^{ed}$
copies of the various $R/Q$ for $Q \in \Lambda$.
\end{quote}
A proof follows in the same manner as the solution to 
  \cite[Exercise 10.4]{HunekeTightClosureBook}, included in the appendix and due to Karen
Smith.  The author is grateful to Craig Huneke for pointing out the
relevance of this
exercise after viewing a preliminary version of this article.
\end{remark}





\section{$F$-Signature}
\label{sec:f-signature}

\subsection{Terminology and key lemmas}
\label{sec:notat-prel-lemm}

\begin{definition}
 Let $(R,\m,k)$ be a reduced \mbox{$F$-finite} local ring of prime characteristic
 $p>0$.   For each $e \in \N$, the \emph{$e$-th Frobenius splitting number of $R$} is 
the largest rank $a_{e} = a_{e}(R)$ of a free $R$-module appearing in a direct sum
decomposition of $F^{e}_{*}R$.  In other words, we may write $F^{e}_{*}R =
R^{ \oplus a_{e}} \oplus M_{e}$ where $M_{e}$ has no free
direct summands.  
\end{definition}

\begin{remark}
We have that $R$ is $F$-pure if and only if $a_{e} > 0$ for some $e
\in \N$, in which case $a_{e} > 0$ for all $e \in \N$. 
If $\widehat{R}$ is the $\m$-adic completion of $R$, then it follows
from $F^{e}_{*}\widehat{R} = \widehat{R} \tensor_{R} F^{e}_{*}R$ that
the $e$-th Frobenius splitting numbers of $R$ and $\widehat{R}$ coincide.
  Since $\widehat{R}$ satisfies the Krull-Schmidt condition
  \cite[Theorem 2.22]{SwanAlgebraicKtheoryLNM}, 
  a direct sum decomposition of $F^{e}_{*}\widehat{R}$ as in the above definition is unique
up to isomorphism.  As a result, the values $a_{e}$ for any
$F$-finite local ring are
independent of the given direct sum decomposition above.  Alternatively, Proposition~\ref{prop:aqandIe} below can be
seen as an elementary proof of this assertion.
\end{remark}

As indicated in the introduction, we aim to show that the
sequence $\{ \frac{a_{e}}{p^{e(d+\alpha(R))}} \}_{e \in \N}$
approaches a limit.  This will be done by applying the uniform
Hilbert-Kunz estimates from the previous section to the following collection of
naturally defined ideals.

\begin{definition}
Suppose $(R, \m, k)$ is an \mbox{$F$-finite} local ring of prime
characteristic $p>0$.
 For each $e \in \N$, we define $I_{e} = \{ r \in R \, | \, \phi(F^{e}_{*}r) \in \m \mbox{
    for all }\phi \in \Hom_{R}(F^{e}_{*}R, R) \}$.
\end{definition}

\begin{lemma}
\label{lem:neededpropsofIe}
 Suppose $(R,\m,k)$ is a reduced \mbox{$F$-finite} local ring with characteristic $p
 > 0$.  Then $I_{e}$ is an ideal containing
 $\m^{[p^{e}]}$.  Furthermore, we have $I_{e}^{[p]} \subseteq I_{e+1}$.
\end{lemma}

\begin{proof}
  We leave it to the
  reader to verify that $I_{e}$ is an ideal of $R$, and also
  $\m^{[p^{e}]} \subseteq I_{e}$ for all $e \in \N$.  Suppose now
  $\phi \in \Hom_{R}(R^{1/p^{e+1}},R)$ and $r \in I_{e}$.  Then it
  follows
\[
\phi((r^{p})^{1/p^{e+1}}) = (\phi|_{R^{1/p^{e}}})(r^{1/p^{e}}) \in \m
\]
as $\phi|_{R^{1/p^{e}}} \in \Hom_{R}(R^{1/p^{e}},R)$ and $r \in
I_{e}$.  Thus, we see $r^{p} \in I_{e+1}$ as desired.
\end{proof}

\begin{proposition}\cite[Corollary 2.8]{AberbachEnescuStructureOfFPure}\cite[Lemma 2.1]{YaoObservationsAboutTheFSignature}
\label{prop:aqandIe}
Suppose $(R,\m,k)$ 
is an \mbox{$F$-finite} local ring with prime
characteristic $p>0$.  If $F^{e}_{*}R = R^{\oplus a_{e}} \oplus M_{e}$ is a
direct sum decomposition where $M_{e}$ has no free direct summands, then
\[
a_{e} = \length_{R}(F^{e}_{*}(R/I_{e})) = \length_{R}\left(
  \Hom_{R}(F^{e}_{*}R,R) / \Hom_{R}(F^{e}_{*}R, \m) \right)
\]
\end{proposition}

{\renewcommand{\v}{\check{v}}
\begin{proof}
  The assumption that $M_{e}$ has no free direct summands implies
  that $\phi(M_{e}) \subseteq \m$ for all $\phi \in
  \Hom_{R}(M_{e},R)$, or equivalently all $\phi \in
  \Hom_{R}(F^{e}_{*}R,R) = \Hom_{R}(R^{\oplus a_{e}},R) \oplus \Hom_{R}(M_{e},R)$.  It is easy to
  see from the definition of $I_{e}$ that
\[
F^{e}_{*}I_{e} = \m^{\oplus a_{e}} \oplus M_{e}
\]
and thus $F^{e}_{*}(R/I_{e}) = F^{e}_{*}R/F^{e}_{*}I_{e} \simeq
k^{\oplus a_{e}}$ has length $a_{e}$.  Similarly, we have
\begin{eqnarray*}
\Hom_{R}(F^{e}_{*}R, \m) & = &
\{ \phi \in \Hom_{R}(F^{e}_{*}R,R) \, | \,
\phi(F^{e}_{*}R)\subseteq \m \} \\ &=& \left(\m
  \Hom_{R}(R^{\oplus a_{e}},R)\right) \oplus \Hom_{R}(M_{e},R)
\end{eqnarray*}
so that also $\Hom_{R}(F^{e}_{*}R,R) / \Hom_{R}(F^{e}_{*}R, \m)
\simeq k \tensor_{R}  \Hom_{R}(R^{\oplus a_{e}},R)$ has
  length $a_{e}$.
\end{proof}
}

\begin{remark}
\label{rmk:injhull}
  The ideals $I_{e}$ appear
  in the works of Y. Yao
  \cite{YaoObservationsAboutTheFSignature} as well as those of I. Aberbach and
  F. Enescu \cite{AberbachEnescuStructureOfFPure}, albeit with a
  different formulation.  For completeness, let us recover their description, which will
  not be needed in the remainder of this article.  We assume $(R,\m,k)$ is
  \mbox{$F$-finite} and complete.  Let $E = E_{R}(k)$ denote the
  injective hull of the residue field, and $(\blank)^{\vee} =
  \Hom_{R}(\blank, E)$ the Matlis duality functor.  If $u \in E$ is a generator for
the socle, we have $k = Ru \subseteq E$. There are
  isomorphisms $E^{\vee} \simeq R$ and $(E/k)^{\vee} \simeq \m$,
  whereby the natural map $(E/k)^{\vee} \to E^{\vee}$ corresponds to
  the inclusion $\m \subseteq R$.  Thus, we have a commutative diagram
  of $F^{e}_{*}R$-modules
\[
\xymatrix{
\Hom_{R}(F^{e}_{*}R,\m) \ar[r] \ar[d]_{\phi} &  \Hom_{R}(F^{e}_{*}R,(E/k)^{\vee})
\ar[r] \ar[d]& (F^{e}_{*}R \tensor_{R} (E/k))^{\vee} \ar[d]^{\psi}\\
\Hom_{R}(F^{e}_{*}R,R) \ar[r] &  \Hom_{R}(F^{e}_{*}R,E^{\vee})
\ar[r] & (F^{e}_{*}R \tensor_{R} E)^{\vee} 
}
\]
where each of the vertical arrows is an inclusion, and the horizontal
arrows are isomorphisms.  By definition, $F^{e}_{*}I_{e} =
\Ann_{F^{e}_{*}R}(\coker(\phi)) = \Ann_{F^{e}_{*}R}(\coker(\psi))$, so
it follows that $F^{e}_{*}I_{e} =
\Ann_{F^{e}_{*}R}(\ker(\psi^{\vee}))$.  Since we have an exact sequence
\[
\xymatrix{
F^{e}_{*}R \tensor_{R} k \ar[r] & F^{e}_{*}R \tensor_{R} E \ar[r]^-{\psi^{\vee}} &
F^{e}_{*}R \tensor_{R} (E/k) \ar[r] & 0
}
\]
we recover the description
\[
I_{e} = \{ r \in R \, | \, F^{e}_{*}r \tensor u = 0 \mbox{ in }
F^{e}_{*}R \tensor_{R} E \} \, \, .
\]
\end{remark}

The following lemma was first observed by
  I. Aberbach and F. Enescu, and a simplified proof has been included
  for completeness.

\begin{lemma}\cite[Theorem 1.1]{AberbachEnescuStructureOfFPure} \cite[Remark 4.4]{SchwedeCentersOfFPurity}
\label{splittingprime}
  Suppose the local ring $(R,\m,k)$ is reduced and \mbox{$F$-finite}. The ideal
\[
P = \bigcap_{e\in \N} I_{e}
\]
is either prime or all of $R$, and is called the \emph{splitting
  prime} of $R$.
\end{lemma}

\begin{proof}
   Supposing $c_{1}, c_{2} \in R \setminus P$, we can find $\phi_{i} \in
  \Hom_{R}(R^{1/p^{e_{i}}},R)$ 
  for some $e_{1}, e_{2} \in \N$ with $\phi_{i}((c_{i})^{1/p^{e_{i}}})
  = 1$ for $i = 1,2$.  But then
\[
\phi := \phi_{1} \circ \phi_{2}^{1/p^{e_{1}}}\left(
c_{1}^{1/p^{e_{1}}-1/p^{e_{1}+e_{2}}}  \cdot ( \, \blank \, ) \right) \in
\Hom_{R}(R^{1/p^{e_{1}+e_{2}}}, R)
\]
satisfies $\phi((c_{1}c_{2})^{1/p^{e_{1}+e_{2}}}) = 1$ so that
$c_{1}c_{2} \in R \setminus P$.
\end{proof}

\begin{remark}
  \label{splittingprimeisstrongfreg}
It is immediate that $P \neq R$ precisely when $R$ is $F$-pure, and
$P = 0$ if and only if $R$ is strongly $F$-regular.  
  Furthermore, it is straightforward to
show $R/P$ is strongly $F$-regular \cite[Theorem
4.7]{AberbachEnescuStructureOfFPure}\cite[Corollary 7.8]{SchwedeCentersOfFPurity}. The prime $P$ is called the \emph{splitting
prime} of $R$, and can also be viewed as the unique largest center of
$F$-purity in $R$; see \cite[Remark 4.4]{SchwedeCentersOfFPurity}.
The \emph{$F$-splitting ratio} appearing below was first introduced in
\cite {AberbachEnescuStructureOfFPure}, though the equivalence of the
original definition with that which follows makes use of Theorem~\ref{thm:sdim}.
\end{remark}

\subsection{Theorem statements and proofs}
\label{sec:theor-stat-proofs}

\begin{theorem}
\label{thm:mainthm}
  Let $(R,\m,k)$ be a $d$-dimensional \mbox{$F$-finite} local ring  with prime characteristic
  $p > 0$ and $\alpha(R) = [k:k^{p}]$. Then the limit
\[
s(R) = \lim_{e \to \infty} \frac{a_{e}}{p^{e(d+\alpha(R))}} 
\]
exists and is called the \emph{$F$-signature of $R$}.  More generally,
if $P$ is the splitting prime of $R$, then 
the limit
\[
r_{F}(R) = \lim_{e \to \infty} \frac{a_{e}}{p^{e(\dim(R/P) + \alpha(R))}}
\]
exists and is called the \emph{$F$-splitting ratio of $R$}.
\end{theorem}

\begin{proof}
If $R$ is not $F$-pure, then $a_{e} = 0$ for all $e \in \N$ and
both statements are clear.  Thus, we may assume $R$ is $F$-pure and hence reduced.

Let us first show the existence of $F$-signature using only the two
properties of the ideals $I_{e}$ shown in Lemma~\ref{lem:neededpropsofIe}.
Since $\m^{[p^{e}]} \subseteq I_{e}$, it follows
from Corollary~\ref{cor:whatineed} that
\[
\lim_{e \to \infty} \left( \frac{1}{p^{ed}} \left(\length(R/I_{e}) -
    e_{HK}(I_{e};R)\right)\right) = 0 \, \, .
\]
From Proposition~\ref{prop:aqandIe}, we see that
\[
\frac{a_{e}}{p^{e(d+\alpha(R))}} = \frac{1}{p^{ed}} \length_{R}(R/I_{e})
\]
for all $e \in \N$.  Thus, to prove the existence of the \mbox{$F$-signature} it suffices to show the sequence $\{ \frac{1}{p^{ed}} e_{HK}(I_{e};R)
\}_{e\in \N}$ approaches a limit.  Since $I_{e}^{[p]} \subseteq
I_{e+1}$, we have 
\[
e_{HK}(I_{e+1};R) \leq e_{HK}(I_{e}^{[p]};R) = p^{d}
\left( e_{HK}(I_{e};R) \right)
\]
 and dividing through by $p^{(e+1)d}$ gives \[
\frac{1}{p^{(e+1)d}} e_{HK}(I_{e+1};R) \leq \frac{1}{p^{ed}} e_{HK}(I_{e};R) \] for all $e \in \N$.  Since
$\{ \frac{1}{p^{ed}} e_{HK}(I_{e};R)
\}_{e\in \N}$
 is non-increasing (and bounded below by zero), the desired conclusion
 follows at once.

More generally for the $F$-splitting ratio, let $\overline{R} = R/P$
and for any ideal $I \subseteq R$ set $\overline{I} =
I\cdot \overline{R}$.  Since
$\m^{[p^{e}]} \subseteq I_{e}$ and $I_{e}^{[p]} \subseteq I_{e+1}$, it
follows immediately that $\overline{\m}^{[p^{e}]} \subseteq
\overline{I_{e}}$ and $\overline{I_{e}}^{[p]} \subseteq
\overline{I_{e+1}}$.  Thus, the preceding argument applied to the
ideals $\overline{I_{e}}$ shows the existence of the limit
\[
\lim_{e \to \infty}
\frac{\length_{\overline{R}}(\overline{R}/\overline{I_{e}})}{p^{e(\dim(R/P))}}
= \lim_{e \to \infty} \frac{e_{HK}(\overline{I_{e}};
  \overline{R})}{p^{e(\dim(R/P))}} \, \, .
\]
Since $a_{e} =
p^{e(\alpha(R))}\length_{R}(R/I_{e}) =
p^{e(\alpha(R))}\length_{\overline{R}}(\overline{R}/\overline{I_{e}})$,
this limit is precisely the $F$-splitting ratio of $R$.
\end{proof}

\begin{remark}
  \label{rmk:notFfinite}
When $(R, \m, k)$ is 
not necessarily \mbox{$F$-finite}, Y. Yao \cite[Remark 2.3]{YaoObservationsAboutTheFSignature}
uses the description from \ref{rmk:injhull} to define the
\mbox{$F$-signature}.  However, the existence of the \mbox{$F$-signature} in this
setting immediately reduces to the \mbox{$F$-finite} case shown
in Theorem \ref{thm:mainthm}.  
\end{remark}

{
\renewcommand{\n}{\mathfrak{n}}
\begin{theorem}
\label{thm:sigineq}
   Let $(R, \m, k)$ be a
 $d$-dimensional $F$-finite characteristic $p > 0$ local
  domain and let $M$ be a finitely generated \mbox{$R$-module.} Denote by $b_{e}$ the maximal rank of a free
  \mbox{$R$-module} appearing in a direct sum decomposition of
  $F^{e}_{*}M$.  Then  
\[
\lim_{e \to \infty} \frac{b_{e}}{p^{e(d+\alpha(R))}} = \rank(M) \cdot
s(R) \, \, .
\]
\end{theorem}

\begin{proof}
Set $I_{e}^{M} = \{ m \in M \, | \, \phi(F^{e}_{*}m) \in \m \mbox{
    for all }\phi \in \Hom_{R}(F^{e}_{*}M, R) \}$.  Following the line
  of argument in Proposition~\ref{prop:aqandIe},  it is easy to see
  that $I_{e}^{M}$ is an $R$-submodule of $M$ with           
\[
b_{e} = \length_{R}(F^{e}_{*}(M/I_{e}^{M}) = p^{e \alpha(R)}
\length_{R}(M/I_{e}^{M}) \, \, .
\]
Furthermore, suppose $m \in M$ and $\phi \in
\Hom_{R}(F^{e}_{*}M,R)$.  Since $\phi( \blank \cdot F^{e}_{*}m ) \in
\Hom_{R}(F^{e}_{*}R, R)$, 
 we have that $\phi(F^{e}_{*}(rm)) \in \m$ for all $r \in I_{e}$.  Thus,
 $rm \in I_{e}^{M}$, and hence 
$I_{e}M \subseteq I_{e}^{M}$.

Let $G \subseteq M$ be a full rank free $R$-submodule, so that there
exists
$0 \neq c \in \Ann_{R}(M/G)$.  We will show $I_{e}^{M} \subseteq (I_{e}M
:_{M} c)$.  As $F^{e}_{*}G \simeq
(F^{e}_{*}R)^{\oplus \rank(M)}$, it is easy to see 
\[
\{ g \in G \, | \, \phi(F^{e}_{*}g) \in \m \mbox{
    for all }\phi \in \Hom_{R}(F^{e}_{*}G, R) \} = I_{e}G \simeq I_{e}^{\oplus \rank(M)}
  \, \, .
\]
Now, suppose we have $\phi \in \Hom_{R}(F^{e}_{*}G, R)$ and $m \in
I_{e}^{M}$.  It follows that $\phi( F^{e}_{*}(cm)) \in \m$ as
$\phi(F^{e}_{*}c \cdot \blank) \in
\Hom_{R}(F^{e}_{*}M,R)$, whence $cM \in I_{e}G \subseteq I_{e}M$.

Using the four term exact sequence
\[
\xymatrix{
0 \ar[r] &  (I_{e}M:_{M} c)/I_{e}M \ar[r] & M/I_{e}M
\ar[r]^-{\cdot c} & M/I_{e}M \ar[r] & M/(I_{e}M + cM) \ar[r] & 0
}
\]
we have that
\[
\length_{R}(I_{e}^{M}/I_{e}M ) \leq \length_{R}( (I_{e}M:_{M}
c)/I_{e}M ) =
\length_{R}(M/(I_{e}M + cM)) \, \, .
\]
Applying Lemma~\ref{lem:lowerdim} now gives
\[
\lim_{e \to \infty} \frac{1}{p^{ed}}(\length_{R}(M/I_{e}M) -
\length_{R}(M/I_{e}^{M})) = 0 \, \, .
\]
Furthermore, by Corollary~\ref{cor:whatineed} we have
\[
\lim_{e \to \infty} \frac{1}{p^{ed}}(\length_{R}(M/I_{e}M) -
e_{HK}(I_{e};M)) = 0 \, \, .
\]
Putting everything together, we now have 
\[
\begin{array}{ccc}
  \lim_{e \to \infty} \frac{b_{e}}{p^{e(d+\alpha(R))}} & = & \lim_{e
    \to \infty} \frac{1}{p^{ed}}\length_{R}(M/I_{e}^{M}) \\ & = & \lim_{e
    \to \infty} \frac{1}{p^{ed}}\length_{R}(M/I_{e}M)  \\
& = & \lim_{e
    \to \infty} \frac{1}{p^{ed}}e_{HK}(I_{e};M) \\ & = & \rank(M) \cdot \lim_{e
    \to \infty} \frac{1}{p^{ed}}e_{HK}(I_{e};R) \\ & = & \rank(M) \cdot s(R)
\end{array}
\]
as desired.
\end{proof}

\begin{remark}
  In the notation of the previous result, Y. Yao has proposed $\lim_{e\to \infty} \frac{b_{e}}{p^{e(d + \alpha(R))}}$
as a potential definition of
\mbox{$F$-signature} for a finitely generated \mbox{$R$-module}~$M$;
see \cite{YaoObservationsAboutTheFSignature}.
\end{remark}

The following Corollary extends 
\cite[Proposition 19]{HunekeLeuschkeTwoTheoremsAboutMaximal} by
removing the Gorenstein
hypothesis.  In particular, this result recovers explicit
formulae (first shown in \cite[Remark 4.7]{YaoObservationsAboutTheFSignature}) for the $F$-signatures of arbitrary finite quotient singularities.

\begin{corollary}
\label{cor:quotsing}
   Let $(R, \m, k) \subseteq (S, \n, l)$ be a local inclusion of
 $d$-dimensional $F$-finite characteristic $p > 0$ local
  domains with corresponding extension of fraction fields $K = \Frac(R)
  \subseteq \Frac(S) = L$.  Assume $S$ is a finitely generated
  $R$-module and write $S = R^{\oplus f}
  \oplus M$ where $M$ has no nonzero free direct summands.
%
  Then
\[
f \cdot s(S) \leq [L:K] \cdot s(R) \, \, .
\]
If additionally $S$ is regular, then equality holds and
\[
s(R) = \frac{f}{[L:K]}
\]
\end{corollary}

\begin{proof}
 Note first that $\alpha(R) = \alpha(S)$ since
$  [l:l^{p}][l^{p}:k^{p}] =[l:k^{p}]= [l:k][k:k^{p}]$.  As above,  let $b_{e}$ denote the maximal rank of a free
  $R$-module appearing in a direct sum decomposition of $F^{e}_{*}S$.
  If we write $F^{e}_{*}S = S^{\oplus a_{e}(S)} \oplus N_{e}$ as
  $S$-modules where $N_{e}$ has no free direct summands, we
  automatically get a direct sum decomposition of $F^{e}_{*}S$
\[
F^{e}_{*}S = (R^{\oplus f} \oplus M)^{\oplus a_{e}(S)} \oplus N_{e}
\]
as an $R$-module with a free direct summand of rank $f \cdot
a_{e}(S)$.  Thus, we have \mbox{$f \cdot a_{e}(S) \leq b_{e}$}.
Furthermore, if $S$ is regular, then $N_{e} = 0$ and equality holds.
Both statements now follow at once by dividing through by $p^{e(d+\alpha(R))} = p^{e(d+\alpha(S))}$ and
letting $e \to \infty$.
\end{proof}

We will need the following results from
\cite{HunekeLeuschkeTwoTheoremsAboutMaximal}; the proof given herein
is due to the original authors.

\begin{theorem}
\label{thm:HLformula}
\cite[Propositions 14 and 15]{HunekeLeuschkeTwoTheoremsAboutMaximal}  Suppose
$(R,\m,k)$ is an \mbox{$F$-finite} local ring with characteristic $p>0$.  If $I \subsetneq
J$ are two $\m$-primary ideals, then
\begin{equation}
\label{eq:6}
\frac{e_{HK}(I;R) - e_{HK}(J;R)}{
\length_{R}(J/I)} \geq s(R) \, \, .
\end{equation}
In particular, if $R$ is Cohen-Macaulay, it follows that
\[
(e(\m;R) - 1)(1 - s(R)) \geq e_{HK}(\m;R) - 1
\]
where $e(\m;R)$ is the Hilbert-Samuel multiplicity of $R$ along $\m$.
\end{theorem}

\begin{proof}
  If we write $F^{e}_{*}R = R^{\oplus a_{e}} \oplus M_{e}$ where
  $M_{e}$ has no direct summands, it follows that
\[
\length_{R}(R/I \tensor_{R} F^{e}_{*}R
) - \length_{R}(R/J \tensor_{R} F^{e}_{*}R) =a_{e}\length_{R}(J/I) +
\length_{R}(JM_{e}/IM_{e}) \geq a_{e} \length_{R}(J/I) \, \, .\]
Dividing through by $\length_{R}(J/I)p^{e(d+\alpha(R))}$ where $d =
\dim(R)$ and letting
$e \to \infty$ gives \eqref{eq:6}.  For the last statement, we may
assume additionally that $R$ is complete with infinite residue
field; see \cite[Remark 2.3(3)]{YaoObservationsAboutTheFSignature}.  Take $J =
\m$ and let $I$ be a minimal reduction of $\m$.  Since $e(\m;R) =
e_{HK}(I;R) = \length_{R}(R/I)$, the desired result follows from \eqref{eq:6}.
\end{proof}

The following result combines 
\cite[Proposition~14]{HunekeLeuschkeTwoTheoremsAboutMaximal}
(Theorem~\ref{thm:HLformula} above) and
Corollary~\ref{cor:quotsing} to
extend \cite[Proposition~19]{HunekeLeuschkeTwoTheoremsAboutMaximal} 
by once again removing the Gorenstein
hypothesis.

\begin{corollary}
  Let $S$ be an $F$-finite regular local ring of characteristic $p$
  and let $G$ be a finite group acting on $S$ with $p \not |  \, \,
  |G|$.  Set $R = S^{G}$ and write $S = R^{\oplus f} \oplus M$ as an
  $R$-module where
  $M$ has no free direct summands.  If $R$ is not regular, then
\[
|G| \geq \frac{f(e(R) -1)}{e(R) - e_{HK}(R)} \, \, .
\]
\end{corollary}

}

\subsection{Open questions and further remarks}

The following Theorem summarizes some further known properties of the
$F$-signature.  Related open questions will follow immediately thereafter.

\begin{theorem}
\label{thm:bigsmeansregular}
Suppose $(R,\m,k)$ is a $d$-dimensional \mbox{$F$-finite} local ring
with prime characteristic $p>0$.
\begin{enumerate}[(i)]
\item \cite[Corollary 16]{HunekeLeuschkeTwoTheoremsAboutMaximal}
We have $s(R) \leq 1$ with equality if and only if $R$ is regular
\item \cite[Theorem 3.1]{YaoObservationsAboutTheFSignature} If 
  $d \geq 2$, then $s(R) \geq
  1-\frac{1}{d!p^{d}}$ if and only if $R$ is regular.
\item \cite[Theorem 0.2]{AberbachLeuschke} We have $s(R) > 0$ if and only
  if $R$ is strongly $F$-regular.
\end{enumerate}
\end{theorem}

\begin{proof} We show only \textit{(i)} and the forward direction of
  \textit{(iii)}, referring the reader to the references above for the remainder.

Assume first that $R$ is not strongly $F$-regular.  Then there exists
$c \in R$ not contained in any minimal prime such that $c \in I_{e}$
for all $e \in \N$.  Since $\m^{[p^{e}]} + \langle c \rangle
\subseteq I_{e}$, we have
\[
s(R) = \lim_{e \to \infty} \frac{a_{e}}{p^{e(d+\alpha(R))}} =
\lim_{e \to \infty} \frac{1}{p^{ed}}\length_{R}(R/I_{e}) \leq
\lim_{e\to\infty} \frac{1}{p^{ed}}\length_{R}(R/\m^{[p^{ed}]}
\tensor_{R} R/\langle c \rangle) = 0
\]
since $\dim(R/\langle c \rangle) < d$.  

Thus, we may assume $R$ is
strongly $F$-regular and hence a domain. 
Let $K = \Frac(R)$ and for
$e \in \N$ write $R^{1/p^{e}} =
R^{\oplus a_{e}} \oplus M_{e}$ where $M_{e}$ has no free direct
summands. Then it follows $K^{1/p^{e}} = K \tensor_{R} R^{1/p^{e}}= K^{\oplus a_{e}} \oplus
(K \tensor_{R} M_{e})$.  Since $[K^{1/p^{e}}:K] = p^{e(d+\alpha(R))}$, we must
have $a_{e} \leq p^{e(d+\alpha(R))}$ and the inequality $s(R) \leq 1$
follows at once.  In case $R$ is regular we have $s(R) = 1$ as $R^{1/p^{e}} = R^{\oplus
  p^{e(d+\alpha(R))}}$ so that $a_{e} = p^{e(d+\alpha(R))}$ for all $e
\in \N$.  Conversely, if $s(R) = 1$ it follows from
Theorem~\ref{thm:HLformula} that $e_{HK}(R) = 1$.  Thus, by
\cite{WatanabeYoshidaHKMultAndInequality}, $R$ is regular.  
\end{proof}

We now conclude by remarking on some of the remaining
open questions concerning the \mbox{$F$-signature}, as well as some forthcoming
results on the $F$-splitting ratio.  The first question, originally
posed by K.-i. Watanabe and K.-i. Yoshida,  is closely
related to the conjectured
  equivalence of strong and weak $F$-regularity.

\begin{question}
  \cite[Question 1.10]{WatanabeYoshidaMinimalHKmultiplicity} If the
  local ring
  $(R,\m,k)$ is reduced and \mbox{$F$-finite} with characteristic $p>0$, can one
  always find $\m$-primary ideals $I \subsetneq J$ such that equality
  holds in \eqref{eq:6}?
\end{question}

\begin{remark}
\label{rmk:minrelHKmult}
  It has been shown by \cite{YaoObservationsAboutTheFSignature} that
  the \mbox{$F$-signature} agrees with the notion of \emph{minimal relative
    Hilbert-Kunz multiplicity} as defined by
  \cite{WatanabeYoshidaMinimalHKmultiplicity}.
\end{remark}

The following question is motivated by
Theorem~\ref{thm:bigsmeansregular} {\itshape (ii)} and the analogous
open question for Hilbert-Kunz multiplicities; see 
\cite{BlickleEnescuOnRingsWithSmallHKMult} and \cite{WatanabeYoshidaHK3dimlocalrings} for further details.

\begin{question}
What is 
\[
\sigma(p,d) = \sup \{ \, s(R) \, |  \mbox{ $R$ is a $d$-dimensional characteristic
  $p$ non-regular local ring } \}
\]
for fixed $p > 0$ and $d$?  For which rings $R$, if any, is the
equality $s(R) = \sigma(p,d)$ achieved?
\end{question}

More generally still, one can ask:

\begin{question}
  What possible values can $s(R)$ attain as $R$ varies
over all local rings of fixed characteristic $p > 0$ and dimension
$d$?  In particular, is $s(R)$ necessarily rational?
\end{question}

\begin{remark}
  At present, all currently known computations of $F$-signature have
  proven to be rational.  These include finite quotient singularities
  (Theorem~\ref{cor:quotsing}), affine semigroup rings \cite{SinghFSignatureOfAffineSemigroup}, as well as
  Segre products and Veronese subrings of polynomial rings
  \cite{WatanabeYoshidaMinimalHKmultiplicity}.  However, a conjecture
  of P. Monsky would give an example of $F$-signature which is irrational.  
\end{remark}

\begin{proposition} Let $R = \mathbb{F}_{2}[[ x, y, z, u, v ]] /
  \langle uv + x^{3} + y^{3} + xyz \rangle$.
  Assuming \cite[Conjecture 1.5]{MonskyIrrHKMultConj}, it follows that
$
s(R) = \frac{2}{3} - \frac{5}{14 \sqrt{7}}
$.
\end{proposition}

\begin{proof}  By \cite[Corollary 2.7]{MonskyIrrHKMultConj}, Monsky's
  conjecture implies $e_{HK}(\m) = \frac{4}{3} + \frac{5}{14
    \sqrt{7}}$ where $\m = \langle x,y,z,u,v\rangle$. Now, consider the parameter ideal  $I = \langle x,y,z, u+v \rangle
  \subseteq R$.  Since $R$ is Cohen-Macaulay, it follows $e_{HK}(I) =
  \length_{R}(R/I) = 2$.  Further, the image of $u$ generates the socle in $R/I$, and we have
  $\langle I, u \rangle = \m$.  Thus, by the proof of \cite[Theorem 11
  (2)]{HunekeLeuschkeTwoTheoremsAboutMaximal}, we have
\[
s(R) = e_{HK}(I) - e_{HK}(\m) = 2 - \left(\frac{4}{3} + \frac{5}{14
    \sqrt{7}}\right) = \frac{2}{3} - \frac{5}{14 \sqrt{7}}
\]
as claimed.
\end{proof}

Another important open question regarding $F$-signature asks how
$F$-signature behaves after localization.  The following observation
is rather immediate: if $(R,\m,k)$ is an $F$-finite local ring and
$\mathfrak{p} \subset \m$ is a prime ideal, then $s(R_{\mathfrak{p}})
\geq s(R)$.  This follows simply by taking direct sum decompositions
of $R^{1/p^{e}}$ and localizing at $\mathfrak{p}$, and one is led  to the
following natural question.

\begin{question}
  \cite{EnescuYao} If $(R,\m,k)$ is an $F$-finite local ring, is the
  function
\[
P \in \Spec(R) \mapsto s(R_{P}) 
\]
is lower semicontinuous?
\end{question}

\begin{remark}
 F. Enescu and Y. Yao have shown that the $e$-th Frobenius splitting
 number function is lower semicontinuous \cite{EnescuYao}.
\end{remark}

Finally, when $R$ is $F$-pure but not strongly $F$-regular, one may
ask to what extent the $F$-splitting ratio $r_{F}(R)$ characterizes the asymptotic behavior of the sequence
of $F$-splitting numbers.  To that end, recall the following
definition from \cite{AberbachEnescuStructureOfFPure}.

\begin{definition}
  \cite{AberbachEnescuStructureOfFPure}
Let $(R, \m, k)$ be an \mbox{$F$-finite} local ring.  The
\emph{$s$-dimension} or Frobenius splitting dimension of $R$ is the
largest integer $k$ such that
\[
\liminf_{e \to \infty} \frac{a_{e}}{p^{e(k+\alpha(R))}}
\]
is not zero, where $a_{e}$ is the $e$-th Frobenius splitting number of $R$.
\end{definition}

\noindent
In future work in preparation \cite{BlickleSchwedeTuckerFsigPairs} together with M. Blickle and K. Schwede,
we are able to show the following.   This gives a positive answer to \cite[Question 4.9]{AberbachEnescuStructureOfFPure}

\begin{theorem}\cite{BlickleSchwedeTuckerFsigPairs}
\label{thm:sdim}
  Suppose $(R, \m, k)$ is a reduced local \mbox{$F$-finite} ring with
  characteristic $p > 0$.  Then the $s$-dimension of $R$ is equal to
  $\dim(R/P)$ where $P$ is the splitting prime of $R$.  In other
  words, if $R$ has an $F$-splitting,  then the $F$-splitting ratio $r_{F}(R)$ is always strictly positive.
\end{theorem}

\bibliographystyle{skalpha}
\bibliography{CommonBib}

\end{document}